\documentclass[10pt]{article}

\usepackage[a4paper,headinclude=false,footinclude=false,DIV=10]{typearea}
\usepackage{setspace}

\usepackage{amssymb,amsthm}

\usepackage{amsmath,amscd}
\usepackage{stmaryrd}

\usepackage[ngerman,english]{babel} 

\usepackage{mathptmx}
\usepackage[scaled]{helvet}
\usepackage[T1]{fontenc}
\newtheoremstyle{mystyle}{}{}{\rmfamily}%
{}{\normalfont\bfseries}{ }{ }{} 

\newtheorem{theorem}{Theorem}[section]
\newtheorem{proposition}[theorem]{Proposition}
\newtheorem{lemma}[theorem]{Lemma}

\theoremstyle{mystyle}

\newcommand{\R}{\mathbb{R}}
\newcommand{\N}{\mathbb{N}}
\renewcommand{\H}{\mathbb{H}}

\newcommand{\cZ}{\mathcal{Z}}
\newcommand{\cF}{\mathcal{F}}

\newcommand{\cW}{\mathcal{W}}
\newcommand{\cL}{\mathcal{L}}
\newcommand{\cP}{\mathcal{P}}
\newcommand{\cK}{\mathcal{K}}

\newcommand{\cH}{\mathcal{H}}
\newcommand{\cV}{\mathcal{V}}

\newcommand{\te}{{\theta}}

\newcommand{\Sig}{{\Sigma}}

\newcommand{\1}{\mathbf{1}}
\newcommand{\with}{:} 

\DeclareMathOperator{\supp}{supp}
\DeclareMathOperator{\id}{id}

\begin{document} 
\selectlanguage{english}
\vspace*{1cm}

\begin{center}
\setstretch{2}
{\huge \bfseries\sffamily An extension of Kesten's criterion for amenability to topological Markov chains}\\

\bigskip \bigskip 
\singlespacing
Manuel Stadlbauer\\ 
\bigskip
\textit{\small 
Departamento de Matemática,\\ 
Universidade Federal da Bahia,\\
Av. Ademar de Barros s/n,\\
40170-110 Salvador (BA), Brasil\\
E-mail: {manuel.stadlbauer@ufba.br}}\\
\bigskip
{\small \today}
\end{center}

\bigskip
\begin{quote} 
\noindent \small \textbf{Abstract.} The main results of this note extend a theorem of Kesten for symmetric random walks on discrete groups to group extensions of topological Markov chains. In contrast to the result in probability theory, there is a notable asymmetry in the assumptions on the base. That is, it turns out that, under very mild assumptions on the continuity and symmetry of the associated potential, amenability of the group implies that the Gurevi\v{c}-pressures of the extension and the base coincide whereas the converse holds true if the potential is Hölder continuous and the topological Markov chain has big images and preimages. Finally, an application to periodic hyperbolic manifolds is given.\\[-.3cm]

\noindent \textbf{Keywords.} Amenability, Group extension, Topological Markov chain, Thermodynamic formalism, Periodic manifold\\ 
\noindent \textbf{MSC 2000.} 37A50, 37C30, 20F69
\end{quote}

\section{Introduction and statement of main results}

The motivation for the analysis of the change of pressure under group extensions stems from the attempt to relate two  classical results from probability theory and geometry on the amenability of discrete groups. The probabilistic result was obtained by Kesten in \cite{Kesten:1959a} and characterises amenability in terms of the spectral radius of the Markov operator associated to a symmetric random walk, that is, a group $G$ is amenable if and only if the spectral radius of the operator acting on $\ell^2(G)$ is equal to 1. The following counterpart in geometry was discovered by Brooks (\cite{Brooks:1985}) using a completely different method. Assume that $G$ is a Kleinian group acting on hyperbolic space $\H^{n+1}$ with exponent of convergence $\delta(G)$ bigger than $n/2$ and that $N \lhd G$ is a normal subgroup. Then the bottoms of the spectra of the Laplacians on $\H/G$ and $\H/N$ are equal if and only if $G/N$ is amenable. Or equivalently, using the characterisation of the bottom of the spectrum in terms of the exponents of convergence, $G/N$ is amenable if and only if 
 $\delta(G)=\delta(N)$. More recently, these results were partially improved. Roblin (\cite{Roblin:2005}) used conformal densities to prove that amenability implies $\delta(G)=\delta(N)$ if $G$ is of divergence type and Sharp obtained in \cite{Sharp:2007} the same statement for convex-cocompact Schottky groups using Grigorchuk's results on the co-growth of shortest representations (see \cite{Grigorchuk:1980}) applied to the Cayley graph of $G$.
  
In here, we consider group extensions of a topological Markov chain for a given potential function. That is, for a topological Markov chain $(\Sig_A,\theta)$, a potential $\varphi:\Sig_A \to (0,\infty)$ and a map $\psi: \Sig_A \to H$ from $\Sig_A$ to a discrete group $H$, the group extension of $(\Sig_A,\theta,\varphi)$ by $\psi$ is defined by 
\[ T: \Sig_A \times H \to \Sig_A \times H, (x,g) \mapsto (\theta(x),g \psi(x))\]
and the lifted potential by $\hat\varphi: \Sig_A \times H \to \R, (x,g) \mapsto \varphi(x)$, where it is throughout assumed that $\psi$ is constant on the states $\cW^1$ of $\Sig_A$. This then gives rise to a natural notion of symmetry through the existence of an involution $\cW^1 \to \cW^1$, $w \mapsto w^\dagger$ such that  $\psi({[w^\dagger]}) = \psi({[w]})^{-1}$ for all states $w \in \cW^1$, where $[w]$ refers to the cylinder associated to $w$.  This involution extends to finite words, leading to the notion of a weakly symmetric potential by requiring that there exists a sequence $(D_n)$ with $\lim_n D_n^{1/n} =1$ such that 
\[  \sup_{x \in [w], y \in [w^\dagger]} \frac{\prod_{j=0}^{n-1}(\varphi\circ\te^j(x))}{\prod_{j=0}^{n-1}(\varphi\circ\te^j(y))} \leq D_n  ,\]
for all $w \in \cW^n$ and with $\cW^n$ referring to the words of length $n$ (see Section \ref{sec:Extensions by amenable groups} for details). Note that this general framework establishes a connection between random walks on groups and the geodesic flow on the unit tangent bundle of $\H/N$ for a certain class of Kleinian groups $G$, since random walks can be recovered by assuming that $\Sig_A$ is a symmetric full shift equipped with a locally constant, symmetric potential whereas the relation to the geodesic flow is obtained through a group extension of the coding map associated with $G$ as considered, e.g.,  in \cite{AaronsonDenker:1999} or \cite{LedrappierSarig:2008}.  

The main results in here extend Kesten's result for random walks to group extensions by replacing statements on the spectral radius by statements on the Gurevi\v{c} pressure $P_G$. Furthermore, they reveal a certain asymmetry with respect to the method of proof and the requirements on the mixing properties of the base transformation $\te$.
The first result, Theorem \ref{theo:Kesten_pt1}, essentially states that, if the potential is weakly symmetric and the group $H$ is amenable, then $P_G(T)=P_G(\te)$. 
This result is a consequence of Kesten's result, since the assumptions give rise to a construction of self-adjoint operators $P_n$ on $\ell^2(H)$, for each $n \in \N$, whose spectral radii have to be equal to $\exp(n P_G(\te))$ as a consequence of Kesten's theorem and the amenability of $H$. 
The arguments for obtaining the converse statement in Theorem \ref{theo:Kesten_reverse} are more intricate and require that the potential is H\"older continuous and summable and that $\te$ has the big images and preimages property. In this situation, we then have that $P_G(T)=P_G(\te)$ implies that $H$ is amenable. The proof is inspired by an argument of Day in \cite{Day:1964} and relies on a careful analysis of the action of the Ruelle operator on the embedding of $\ell^2(H)$ into a certain subspace of $C(\Sig_A \times H)$. 

Finally, it is worth noting that the results of Kesten and Grigorchuk mentioned above can be easily deduced from Theorems \ref{theo:Kesten_pt1} and \ref{theo:Kesten_reverse} by choosing $\Sig_A$ to be a shift in the generators and $\varphi$ to be a suitable, locally constant, symmetric potential. Furthermore, a reformulation in terms of Gibbs-Markov maps reveals that the results in here might be seen as an extension of Kesten's theorem for random walks on groups with independent, identically distributed increments to random walks with stationary, exponentially $\psi$-mixing increments. This generalisation then allows to apply the results to normal covers of Kleinian groups. That is, the theorem of Brooks extends to the class of essentially free Kleinian groups as defined below and, in particular, to a class of Kleinian groups with parabolic elements of arbitrary rank. 

While writing this article, related results for locally constant potentials were independently obtained by Jaerisch (\cite{Jaerisch:2011}) where it is shown that under this restriction, Theorem \ref{theo:Kesten_reverse} is a consequence of a version of Kesten's result for graphs in \cite{OrtnerWoess:2007}.

\section{Topological Markov chains}

For a countable alphabet $I$ and a matrix $A= (a_{ij}:\; i,j \in I)$ with $a_{ij}\in \{0,1\}$ for all $i,j \in I$ and $\sum_{j} a_{ij} >0$ for all $i\in I$, let the pair $(\Sig_A, \te)$ denote the associated one-sided topological Markov chain. That is, 
\begin{align*}
\Sig_A& := \left\{(w_k:\;k=0,1,\ldots)\with w_k \in I, a_{w_k w_{k+1}}= 1 \;\forall i = 0,1, \ldots  \right\},\\
\te &: \Sig_A \to \Sig_A, \te:(w_k:\;k=1,2,\ldots) \mapsto (w_k:\;k=2,3,\ldots).
\end{align*} 
A finite sequence $w=(w_1\ldots w_n)$ with $n \in \N$, $w_k \in I$ for $k=1,2, \ldots, n$ and $a_{w_k w_{k+1}}= 1$ for $k=1,2, \ldots, n-1$ is referred to as a \emph{word of length $n$}, and the set 
\[[w]:= \{(v_k)\in \Sig_A \with w_k=v_k\; \forall k=1,2,\ldots, n\}\]
as a \emph{cylinder of length $n$}. The \emph{set of admissible words of length $n$} will be denoted by $\cW^n$, the length of $w \in \cW^n$ by $|w|$ and the set of all admissible words by $\cW^\infty = \bigcup_n \cW^n$. 
Furthermore, since $\te^n: [w]\to \te^n([w])$ is a homeomorphism, the inverse exists and will be denoted by $\tau_w: \te^n([w]) \to [w]$. 
 For  $a,b \in \cW^\infty$ and $n \in \N$ with $n \geq |a|$ , set
\begin{align*}
\cW_{a,b}^n &= \{(w_1 \ldots w_n)\in \cW^n \with (w_1\ldots w_{|a|}) =a,\; w_nb \hbox{ admissible}\}.
\end{align*}
As it is well known, $\Sig_A$ is a Polish space with respect to the topology generated by cylinders, and $\Sig_A$ is compact and locally compact with respect to this topology if and only if $I$ is a finite set. 
Furthermore, recall that $\Sig_A$ 
is called \emph{topologically transitive} if for all $a,b \in I$, there exists $n_{a,b}\in \N$ such that  $\cW_{a,b}^{n_{a,b}} \neq \emptyset$ and that $\Sig_A$ 
is called \emph{topologically mixing} if for all $a,b \in I$, there exists $N_{a,b}\in \N$ such that $\cW_{a,b}^n \neq \emptyset$ for all $n \geq N_{a,b}$. 
Moreover, a topological Markov chain is said to have \emph{big images} or \emph{big preimages}
if there exists a finite set $\mathcal{I}_{\textrm{\tiny bip}}  \subset \cW$ 
such that for all $v \in \cW$, there exists $\beta \in \mathcal{I}_{\textrm{\tiny bip}} $ such that $(v \beta) \in \cW^2$ or 
$(\beta v) \in \cW^2$, respectively. Finally, a topological Markov chain is said to have the \emph{big images and preimages (b.i.p.) property}  if the chain is topologically mixing
and has big images and preimages (see \cite{Sarig:2003a}). Note that the b.i.p. property coincides with the notion of finite irreducibility for topological mixing topological Markov chains as introduced by Mauldin and Urbanski (\cite{MauldinUrbanski:2001}).

We now consider a pair $(\Sig_A,\varphi)$ where $\varphi: \Sigma_A\to \R$ is a strictly positive function which we refer to as a \emph{potential}. For $n \in \N$ and $w \in \cW^n$, set 
\begin{equation}\label{def:bounded_distortion}
\Phi_n := \prod_{k=0}^{n-1} \varphi \circ \te^k \hbox{ and } C_w:= \sup_{x,y \in [w]} \Phi_n(x)/\Phi_n(y).\end{equation}
The potential  $\varphi$ is said to have \emph{(locally) bounded variation} if $\varphi$ is continuous and there exists $C>0$ such that  $C_w\leq C$ for all $n \in \N$ and $w \in \cW^n$, and is called potential of \emph{medium variation} if $\varphi$ is continuous and, for all $n \in \N$, there exists $C_n>0$ with  $C_w\leq C_n$ for all $w \in \cW^n$ and $ \lim_{n\to \infty} \sqrt[n]{C_n} =1$. 
For positive sequences $(a_n),(b_n)$ we frequently will write $a_n \ll b_n$ if there exists $C>0$ with $a_n \leq C b_n$ for all $n \in \N$, and $a_n \asymp b_n$ if $a_n \ll b_n \ll a_n$. 
A further, stronger assumption on the variation is related to local Hölder continuity. Therefore, recall that the $n$-th variation of a function $f: \Sig_A \to \R$ is defined by
\[V_n(f)=\sup\{ |f(x)-f(y)|: x_i=y_i,\, i=1,2,\ldots,n \}.\]
The function $f$ is referred to as a \emph{locally H\" older continuous function} if there exists
$0<r<1$ and $C\geq 1$ such that $V_n(f)\ll r^n$ for all $n\geq 1$. Moreover, we refer to a locally H\" older continuous function with $\|f\|_\infty < \infty$ as a H\" older continuous function. We now recall the following well-known estimate.
For $n \leq m$, $x,y \in [w]$ for some $w \in \cW^m$, and a locally Hölder continous function $f$,  
\begin{align} \label{eq:Hoelder_estimate}
\big|\sum_{k=0}^{n-1} f\circ \te^k(x) - f\circ \te^k(y) \big| 
\ll  \frac{1}{1-r} r^{m-n}. 
\end{align}
In particular, the function $\exp f$ is a potential of bounded variation. For a given potential $\varphi$, the basic objects of thermodynamic formalism are partition functions. Since the state space might be countable, we consider partition functions $Z_a^n$ for a fixed $a \in I$ which are  defined by
\[Z_a^n := \sum_{ \te^n(x) =x,\; x \in [a]} \Phi_n(x).\]
Furthermore, we refer to the exponential growth rate of $Z_a^n$, that is to 
\[P_G(\te,\varphi) :=  \limsup_{n \to \infty} \log \sqrt[n]{Z_a^n} = \limsup_{n \to \infty} \frac{1}{n} \log {Z_a^n},\]
as the \emph{Gurevi\v{c} pressure} of $(\Sig_A,\te,\varphi)$. This notion was introduced in \cite{Sarig:1999} for topologically mixing systems where $\log \varphi$ is locally Hölder continuous. 
If $(\Sig_A,\te,\varphi)$ is transitive and $\varphi$ is of medium variation, arguments in there combined with the decomposition of $\te^p$ into mixing components, where  $p$ stands for the period of $(\Sig_A,\te)$, show that $P_G(\te,\varphi)$ is independent of the choice of $a$ and that 
\[P_G(\te,\varphi)=\lim_{n\to \infty,\; \cW^n_{a,a} \neq \emptyset}\frac{1}{n} \log {Z_a^{n}}.\]
Furthermore, it is easy to see that $P_G(\te,\varphi)$ remains unchanged by replacing $a \in \cW^1$ with some  $a \in \cW^n$. 
Also recall that, if $\log \varphi$ is Hölder continuous and the system is topologically mixing, then a variational principle holds (see \cite{Sarig:1999}). 

We now recall the definitions of conformal and Gibbs measures related to a given potential $\varphi$.  A Borel probability measure $\mu$ is called \emph{$\varphi$-conformal} if 
\[ \mu(\te(A)) = \int_A \frac{1}{\varphi}d\mu  \] 
 for all Borel sets $A$ such that $\te|_A$ is injective. For $(w_1\ldots w_{n+1}) \in \cW^{n+1}$ and a potential of medium variation, it then immediately follows that  
\begin{equation}\label{eq:conformal_estimate} C_n^{-1} \mu(\te([w_{n+1}])) \leq \frac{\mu([w_1\ldots w_{n+1}])}{\Phi_n(x)} \leq C_n \mu(\te([w_{n+1}]))\end{equation} 
for all $x \in [w_1\ldots w_{n+1}]$. Note that this estimate implies that $P_G(\te,\varphi)=0$ is a necessary condition for the existence of a conformal measure  with respect to a potential of medium variation. Moreover, the above estimate motivates the following definitions. Assume that there exists a sequence $(B_n:n \in \N)$ with $B_n \geq 1$ such that
\begin{equation}\label{eq:Gibbs} B_n^{-1}\leq \frac{\mu([w])}{\Phi_n (x)}<B_n \end{equation} 
for all $n \in \N$, $w \in \cW^n$ and $x \in [w]$. If $\sup_n B_n < \infty$, then $\mu$ is called \emph{$\varphi$-Gibbs measure}, and if $\lim_{n\to \infty} B_n^{1/n}=1$, then $\mu$ is called \emph{weak $\varphi$-Gibbs measure}. In order to introduce a further basic object, the Ruelle operator, we define the action of the inverse branches of $\tau_v$ on functions as follows. For $v \in \cW^n$ and $f : \Sigma_A \to \R$, set 
\[f\circ \tau_v : \Sig_A \to \R, \  x \mapsto \1_{\te^n([v])}(x) f(\tau_v(x)),\]
that is $f\circ \tau_v(x) := f(\tau_v(x))$ for  $x \in \te^n([v])$ and $f\circ \tau_v(x) := 0$ for  $x \notin \te^n([v])$. The \emph{Ruelle operator} is then defined by 
\[ L_\varphi (f) = \sum_{v \in \cW^1} \varphi\circ\tau_v \; \cdot \;f\circ\tau_v, \]
where $f : \Sigma_A \to \mathbb{C}$ is an element of an appropriate function space such that the possibly infinite sum on the right hand side is well defined.

\section{Group extensions of topological Markov chains} \label{sec:Extensions by amenable groups}
To introduce the basic object of our analysis, fix a countable discrete  group $G$ and a map $\psi: \Sig_A \to G$ such that $\psi$ is constant on $[w]$ for all $w \in \cW^1$. Then, with $X:= \Sig_A \times G$ equipped with the product topology, 
the \emph{group extension or $G$-extension} $(X,T)$ of $(\Sig_A, \te)$ is defined by
\[ T: X\to X, (x,g) \mapsto (\te x, g \psi(x)).\]
Observe that $(X,T)$ also is a topological Markov chain and that its cylinder sets are given by $[w,g]:=[w]\times \{g\}$, for $w \in \cW^\infty$ and $g \in G$. Furthermore, set $X_g:= \Sig_A\times \{g\}$ and  
 \[\psi_{n}(x) := \psi(x)\psi(\te x) \cdots \psi(\te^{n-1}x)\]
  for $n \in \N$ and $x \in \Sig_A$. Observe that $\psi_{n}:\Sig_A \to G$ is constant on cylinders of length $n$ and, in particular, that $\psi(w) := \psi_{n}(x)$, for some $x \in [w]$ and $w \in \cW^n$, is well defined. Moreover, for $a,b \in \cW^1$ and  $n \in \N$, set
\[G_n(a,b):= \{ \psi(w) \with n \in \N, w\in \cW^n, [a]\supset[w], \te^n([w]) \supset [b]\}.\]
 Note that $(X,T)$ is {topologically transitive} if and only if, for $a,b \in \cW^1$, $g\in G$, there exists $n \in \N$ with $g\in G_n(a,b)$, and that $(T,X)$ is  \emph{topologically mixing} if and only if,  for $a,b \in \cW^1$ and  $g\in G$, there exists $N \in \N$ (depending on $a,b,g$) such that  $g\in G_n(a,b)$ for all $n>N$. Note that the base transformation $(\te,\Sig_A)$ of a topologically transitive group extension has to be topologically transitive, and topologically mixing if the extension is topologically mixing, respectively. Furthermore, if  $(T,X)$ is topologically transitive, then $\{\psi(a)\with a \in \cW^1\}$ is a generating set for $G$. 

Throughout, we now fix a topological mixing topological Markov chain $(\Sig_A, \te)$, and a topological transitive $G$-extension $(T,X)$. Furthermore, we fix a (positive) potential $\varphi:\Sig_A \to \R$ with $P_G(\te,\varphi)=0$. Note that $\varphi$ lifts to a potential  $\hat\varphi$ on $X$ by setting $\hat\varphi(x,g):= \varphi(x)$. For ease of notation, we will not distinguish between $\hat\varphi$ and $\varphi$. Moreover, for $v \in \cW^\infty$, the inverse branch given by $[v,\cdot]$ will be as well denoted by $\tau_v$, that is $\tau_v(x,g) := (\tau_v(x),g\psi(v)^{-1})$.
In order to distinguish between the Ruelle operator and the partition functions with respect to $\te$ and $T$, these objects for the group extension will be written in calligraphic letters, that is, for $a \in \cW$, $\xi \in [a]\times \{\id\}$, $(\eta,g)\in X$, and $n \in \N$,  
\begin{align*}
\cL(f)(\xi,g)  := \sum_{v \in \cW} \varphi\circ\tau_v(\xi) f\circ\tau_v(\xi,g), \quad 
\cZ_{a,g}^{n} := \sum_{ \genfrac{}{}{0pt}{}{T^n(x,g) =(x,g),}{x \in [a]}} \Phi_n(x).
\end{align*}

\section{Extensions by amenable groups} 
In this section, we show that the Gurevi\v{c} pressure remains unchanged under extension by an amenable group. In particular, it will turn out that this statement is true under very mild conditions. Also note that a similar result was proven in \cite{Sharp:2007} for 
extensions of subshifts of finite type with respect to a hyperbolic potential.
We first recall the definition of amenability using F\o lner's condition (see \cite{Folner:1955}). That is, $G$ is referred to as an \emph{amenable group} if and only if there exists a sequence $(K_n)$ of finite subsets of $G$ with $\bigcup_n K_n = G$ such that 
  \[ \lim_{n \to \infty} |g K_n \triangle K_n|/ |K_n|=0 \quad \forall g \in G.\]
In here, $ \triangle$ refers to the symmetric difference, and $|\cdot|$ to the cardinality of a set. We are now interested in the characterisation of amenability in terms of the  Gurevi\v{c} pressure of a group extension. As it  will turn out, this is an extension of the following result of Kesten in \cite{Kesten:1959a}.  Let $m$ be a probability measure on $G$ with $m(g^{-1}) = m(g)$ and assume that the support $\supp(m)$ of $m$ is a generating set for $G$. Then the spectral radius of the operator $P$ on the complex space $\ell^2(G)$ given by $Pf(\gamma) := \sum_{g \in G} f(\gamma g^{-1}) m(g)$ is equal to one if and only if $G$ is amenable.

 In order to obtain an extension of this result to shift spaces, one has to consider group extensions with symmetry. Namely, we say that $(\Sig_A,\te,\psi)$ is \emph{symmetric} if there exists  $\cW^1 \to \cW^1$, $w \mapsto w^\dagger$ with the following properties.
\begin{enumerate}
  \item  For $w \in \cW^1$, $(w^\dagger)^\dagger=w$.
  \item  \label{def:2} For $v,w \in \cW^1$, the word $(vw)$ is admissible if and only if $(w^\dagger v^\dagger)$ is admissible.
  \item  $\psi(v^\dagger) = \psi(v)^{-1}$ for all $v \in \cW^1$.
\end{enumerate}
Observe that this notion of symmetry can be extended to an involution of $\cW^\infty$ by defining $(w_1\ldots w_n)^\dagger := (w_n^\dagger\ldots w_1^\dagger)$. We refer to $\varphi$ as a \emph{weakly symmetric potential} if $\varphi$ is continuous and there exists a sequence $(D_n)$ with $ \lim_{n\to \infty} \sqrt[n]{D_n} =1$ such that, for all $n \in \N$ and $w  \in \cW^n$,  
\[ \sup_{x \in [w], y \in [w^\dagger]} \Phi_n(x)/\Phi_n(y) \leq D_n.\]
 If $\sup D_n < \infty$, then $\varphi$ is referred to as a \emph{symmetric potential}. Note that a weakly symmetric  or symmetric potential is necessarily of medium variation with respect to $C_n:= D_n^2$ or of bounded variation, respectively.

We now prove that  amenability of $G$ implies that $P_G(T)=P_G(\te)$. For this purpose, we construct  a family of self-adjoint operators on $\ell^2(G)$ for symmetric group extensions as follows. By transitivity, there exists $a \in \cW^\infty$ with $\psi(a) = \id$. Furthermore, for $v \in \cW^\infty$, we have that $a v a^\dagger$ is admissible if and only if $(a v a^\dagger)^\dagger = a v^\dagger a^\dagger$ is admissible. Hence, for $n > |a|$,
\[\iota: \cW^{n}_{a,a^\dagger} \to \cW^{n}_{a,a^\dagger},\  \iota(av):=av^\dagger\]
 defines an involution. Since $\psi(a)=\id$, it follows that $\psi(v) = \psi(\iota v)^{-1}$ for all $v \in \cW^n_{a,a^\dagger}$.   
We now assume that $P_G(\te)$ is finite and fix $\xi \in [a^\dagger]$. For $n \in \N$ and $v \in \cW^n_{a,a^\dagger}$, set
 \[\pi{(\xi,v)} := \frac{1}{2}\left( \Phi_n(\tau_v(\xi)) +  \Phi_n(\tau_{\iota v}(\xi)\right).\] 
This then gives rise to an operator $P_n : \ell^2(G) \to \ell^2(G)$ on the complex Hilbert space $\ell^2(G)$ by
\[P_n(f)(\gamma) := \sum_{v \in \cW^n_{a, a^\dagger}}  \pi{(\xi,v)} f(\gamma \psi(v)^{-1}),\]
where, for ease of notation, the fact that the operator $P_n$ depends on $\xi$ is omitted.
Note that  $P_n(1) =  L^n_\varphi(\1_{[a]})(\xi) < \infty$ and, with $\langle f,g\rangle = \sum \overline{f(\gamma)}g(\gamma)$ referring to the standard inner product, $\langle \1_\gamma,P_n(\1_{\gamma^\ast})\rangle  =  \langle P_n(\1_{\gamma}),\1_{\gamma^\ast}\rangle$, for  $\gamma, \gamma^\ast \in G$. In particular, this implies that $P_n$ is self-adjoint.  

Combining $\langle \1_\gamma,P_n(\1_{\gamma})\rangle = \langle \1_{\id},P_n(\1_{\id})\rangle$ for all $\gamma \in G$ with $P_n$ being self-adjoint then gives that the spectral radius $\rho_n$ of $P_n$ satisfies (see, e.g., \cite{KaimanovichVershik:1983})  
\[ \rho_n = \limsup_{k \to \infty} \sqrt[k]{\langle \1_{\id},P^k_n(\1_{\id})\rangle},\]
and that, if $P_n$ is a positive operator, then the $\limsup$ above is a limit. As an immediate corollary to Kesten's theorem on the spectral radius of the Markov operator associated with a symmetric random walk (see \cite{Kesten:1959a} and \cite[Theorem 5]{KaimanovichVershik:1983}), we obtain the following theorem for group extensions under very mild conditions.

\begin{theorem}\label{theo:Kesten_pt1}
Assume that $T$ is a topologically transitive, symmetric group extension of the topologically mixing topological Markov chain $(\Sigma_A,\te)$, $\varphi$ is  weakly symmetric and $P_G(\te)$ is finite. 
Then $P_G(T)=P_G(\te)$, if $G$ is an amenable group.
\end{theorem}

\begin{proof} Since we obviously have that $P_G(T) \leq P_G(\te)$, it remains to prove the reverse inequality. For $n \geq |a|$ and $g \in G$, set 
\[m_n(g):= \frac{1}{P_n(1)}\sum_{v \in \cW^n_{a,a^\dagger}: \ \psi(v)=g} \pi{(\xi,v)}. 
\]
Note that $m_n(g)$ is well defined since $|P_G(\te)|<\infty$ and that $m_n$ is a symmetric probability measure on $G$, that is $m_n(g)=m_n(g^{-1})$. Moreover, note that the group generated by the support of $m_n$ is a subgroup of $G$ and hence is amenable.  
Since the Markov operator associated with the symmetric random walk given by $m_n$ coincides with $P_n(\cdot)/P_n(1)$ we have by Kesten's theorem that $\rho_n = P_n(1)$. In order to prove the assertion it suffices to show that  $\lim_n \log (\rho_n)/n = P_G(\te)$ and $\limsup_n \log (\rho_n)/n \leq P_G(T)$.

\noindent{\textsc{Step 1.}} For $n \geq |a|$, we have $\iota w \in \cW^n_{a,a^\dagger}$ if and only if $w \in \cW^n_{a,a^\dagger}$. Hence, 
\begin{align*}
 \rho_n = P_n(1) =  \sum_{w \in \cW^n_{a,a^\dagger}} \frac{1}{2} \left(\Phi_n(\tau_w(\xi)) + \Phi_n(\tau_{\iota w}(\xi))\right) =  
 \sum_{w \in \cW^n_{a,a^\dagger}}  \Phi_n(\tau_w(\xi)).
 \end{align*}
 Choose $k \geq |a|$ and $v \in \cW^k_{a,a^\dagger}$. For $n \geq k$, the medium variation property of $\varphi$ gives  
\begin{align*}
 \rho_n =  \sum_{w \in \cW^n_{a,a^\dagger}}  \Phi_n(\tau_w(\xi))
\geq   C_k^{-1} C^{-1}_{n} \Phi_k(\tau_v(\xi))\, Z^{n-k}_b,
\end{align*}
where $b := (a_{|a|})^\dagger$ for $a =(a_1\ldots a_{|a|})$. Hence, $\limsup_n (\log \rho_n)/n = P_G(\te)$.
 
\noindent{\textsc{Step 2.}} For ease of notation, we will write $x=y^{\pm k }z$ for $y^{- k }z\leq x \leq y^{ k }z $. Then, for $w \in \cW^{n}_{a,a^\dagger}$, weak symmetry  implies
\[ \Phi_n(\tau_w(\xi))  = \Phi_{n+|a|}(\tau_w(\xi))/\Phi_{|a|}(\xi) 
=  (C_nD_{n+|a|})^{\pm 1} \Phi_n(\tau_{\iota w}(\xi)).\]
Hence, there exists $(\widehat{D}_n)$ with $\pi{(\xi,w)} = \Phi_n(\tau_w(\xi)) \widehat{D}_n^{\pm 1}$ and  $\lim_{n} {\widehat{D}_n}^{1/n} =1$.  
By transitivity, there exist $l \in \N$, and $v \in \cW^{l}_{a^\dagger a}$ with $\psi(v)=\id$. Hence, for $w_1, w_2, \ldots, w_k \in \cW_{a a^\dagger}^n$, we have 
$w^\ast := (w_1  v w_2 v \ldots w_k v) \in  \cW^{k(n+l)}_{aa}$ and, for $x \in [w^\ast]$, 
\begin{align*} \Phi_{k(n+l)}(x) & \geq \left( \widehat{D}_{n}^{-1} \inf\{\Phi_l(y)\with y \in [va] \}\right)^{k} \prod_{i=1}^k\pi{(\xi,w_i)}. 
 \end{align*}
This gives rise to the following estimate for $\langle 1_{\id}, P_n^k(1_{\id}) \rangle$ in terms of the partition function $\cZ_{a,\id}^{k(n+l)}$ with respect to $T$ as defined above. 
\begin{align*}
 \frac{1}{nk} \log \langle 1_{\id}, P_n^k(1_{\id}) \rangle  & \leq   \frac{1}{n}  \log \left(\widehat{D}_n^{-1}  \inf\{\Phi_l(y)\with y \in [va] \} \right) +    \frac{1}{nk} \log  \cZ_{a,\id}^{k(n+l)}.
\end{align*}
Since $P^2_n$ is a positive operator with $\|P^2_n\| = \rho_n^2$, we obtain, 
by taking the limit for $k \to \infty$, $k \in 2\N$ and then for $n \to \infty$, that $\lim_n \log (\rho_n)/n \leq P_G(T)$.
~ \hfill $\square$ \end{proof}

\section{Kesten's theorem for group extensions}

The essential ingredient of the proof of Theorem \ref{theo:Kesten_pt1} is the fact that a symmetric probability measure defines a symmetric operator on $\ell^2(G)$. So, in order to prove the analogue of Kesten's result for group extensions of topological Markov chains, it remains to show that $P_G(\te)=P_G(T)$ implies amenability, where one is tempted again to use the spectral radius formula applied to symmetric operators on $\ell^2(G)$ as, e.g., in  \cite[p. 478]{KaimanovichVershik:1983}. 
However, it will turn out that the key step in here is to carefully analise an embedding of $\ell^2(G)$ and use an argument based on uniform rotundity to show that in case of amenable groups, almost eigenfunctions are indicator functions.

The arguments of proof in here rely on stronger topological mixing properties in the base. That is, we will have to assume that the base has the b.i.p.-property. As a first consequence of this property, we obtain the existence of the following finite subset of $\cW^\infty$. 

\begin{lemma} \label{lem:finite_subsets} Assume that $(X,T)$ is a topologically transitive group extension of $(\Sig_A,\te)$ and that $(\Sig_A,\te)$ has the b.i.p.-property. Then there  exists $n \in \N$ and a finite subset $\mathcal{J}$ of $\cW^n$ such that for each pair $(\beta,\beta')$ with $\beta,\beta' \in \mathcal{I}_{\textrm{\tiny bip}}$ there exists $w_{\beta,\beta'}  \in  \mathcal{J}$ such that
 $(w_{\beta,\beta'})\in \cW^{n}$ and $\psi_n(w_{\beta,\beta'}) = \id$.   
\end{lemma}

\begin{proof} 
Let $p \in \N$ refer to the period $p$ of the transitive topological Markov chain $(X,T)$.
Then, for each  $a\in \cW\times G$, there exists $N_{a} \in \N$ such that $T^{ps}([a,g]) \supset [a,g]$ for all $s \geq N_a$ and $g \in G$.
Hence, for each pair $(\beta,\beta') \in  \mathcal{I}_{\textrm{\tiny bip}}^2$ with $(\beta' \beta)$ admissible, there is $N_{\beta,\beta'}$ such that for each $s \geq N_{\beta,\beta'}$ there exists $v_{\beta,\beta'} \in \cW^{ps-2}$ such that $(\beta' \beta v_{\beta,\beta'} \beta')$ is admissible 
 and $\psi_{ps}(\beta' \beta v_{\beta,\beta'}) = \psi_{ps}( \beta v_{\beta,\beta'} \beta') = \id$. Since $ \mathcal{I}_{\textrm{\tiny bip}}$ is finite it follows that there exists $k$ (given by $\max \{p N_{\beta,\beta'}\with (\beta,\beta') \in  \mathcal{I}_{\textrm{\tiny bip}}^2 \}$) such that $v_{\beta,\beta'}$ can be always chosen to be an element of $\cW^{k-2}$.

 By possibly adding finitely many states we may assume without loss of generality that the subsystem of $\Sig_A$ with alphabet $\mathcal{I}_{\textrm{\tiny bip}}$  is topologically mixing. It then follows from this that there exists some $l \in \N$ such that each pair $(\beta_0,\beta_l)$ in $\mathcal{I}_{\textrm{\tiny bip}}$ can be connected by an admissible word of the form 
 \[ w_{\beta_0,\beta_l}:= (\beta_0 v_{\beta_0,\beta_1} \beta_1 \beta_2 v_{\beta_2,\beta_3} \beta_3 \beta_4 \ldots  v_{\beta_{l-1},\beta_l} \beta_l ).\]  
The assertion follows with $ \mathcal{J}:= \{ w_{\beta,\beta'} \with \beta,\beta'\in \mathcal{I}_{\textrm{\tiny bip}} \}$.
~ \hfill $\square$ \end{proof}
A further important consequence of the b.i.p.-property is the existence of an invariant Gibbs measure. That is, if  $(\Sigma_A,\te)$ and $\varphi$ are given such that $(\Sigma_A,\te)$ has the b.i.p.-property, $\log \varphi$ is Hölder continuous and $\|L_\varphi(1)\|_\infty < \infty$, then there exists a $\exp(-P_G(\te,\varphi))\cdot \varphi$-Gibbs measure $\mu$ and a Hölder-continuous eigenfunction $h$ such that $h d\mu$ is an invariant probability measure. It is also worth noting that  $(\Sigma_A,\te)$ has the Gibbs-Markov property as defined below with respect to $\mu$ (see \cite{MauldinUrbanski:2001,Sarig:2003a}). Moreover, since the function $\log h$ is uniformly bounded from above and below, we assume from now on, without loss of generality, that $ P_G(\te,\varphi)=0$ and $L_\varphi(1) = 1$.

The existence of $\mu$ then gives rise to the following definition of $\cH_1$ and  $\cH_\infty$.
Given a measurable function $f:X \to \R$, $g \in G$ and $p=1$ or $p=\infty$, set $\|f\|_p^g := \|f(\cdot\ ,g)\|_p$ and define
\begin{align*} 
 \llbracket f\rrbracket_p:= \sqrt{{\textstyle\sum_{g \in G} (\|f\|_p^g)^2 } } \quad \hbox{ and } \quad   \cH_p:= \{f: X \to \R\with \llbracket f \rrbracket_p < \infty \}.
\end{align*}
Furthermore, set $\cH_c := \{f \in \cH \with f \hbox{ is constant on } {X_g} \forall g \in G\}$ and $\rho:= \exp(P_G(T,\hat\varphi))$.

\begin{proposition}\label{prop:operator_on_cH} The function spaces $(\cH_1, \llbracket \cdot \rrbracket_1)$ and $(\cH_\infty, \llbracket \cdot \rrbracket_\infty)$ are Banach spaces, the operators $\cL_\varphi^k: \cH_\infty \to \cH_\infty$ are bounded and there exists $C\geq1$ such that   $\llbracket\cL_\varphi^k\rrbracket_\infty \leq C$ for all $k\in \N$. Furthermore, 
\[ \Lambda_k := \sup\left(\left\{  \llbracket \cL_\varphi^k(f)  \rrbracket_1 /  \llbracket f \rrbracket_1  \with f \geq 0, f\in \cH_c  \right\}\right) \leq 1 \] 
and $\limsup_{k \to \infty} (\Lambda_k)^{1/k} \geq \rho$.
\end{proposition}
\begin{proof} The proof that $(\cH_p, \llbracket \cdot \rrbracket_p)$ are Banach spaces is standard and therefore omitted.
 In order to prove the uniform bound for $\llbracket\cL_\varphi^k\rrbracket_\infty$, assume that $f \in \cH_\infty$ and $k \in \N$. Using Jensen's inequality, we obtain
\begin{align} 
\nonumber 	\llbracket  \cL^k_{\varphi}(f) \rrbracket_\infty^2  
& \leq \sum_{g\in G} \sup_{x \in \Sig_A} \left(   {\sum_{v \in \cW^k}} \Phi_k \circ \tau_v(x) 
\|f\|^{g\psi_k(v)^{-1}}_\infty  \right)^2 
\leq \sum_{g\in G} \sup_{x \in \Sig_A}   {\sum_{v \in \cW^k}} \Phi_k \circ \tau_v(x) 
\left( \|f\|^{g\psi_k(v)^{-1}}_\infty  \right)^2 \\
\nonumber  & 
\leq  C  {\sum_{v \in \cW^k}} \mu([v]) \llbracket  f \rrbracket_\infty^2 = C \llbracket  f \rrbracket_\infty^2,
\end{align}
where $C$ is given by the Gibbs property of $\mu$ on $\Sig_A$ (see (\ref{eq:Gibbs})). Hence, $C$ is an upper bound for $\llbracket \cL_\varphi^k \rrbracket_\infty$, independent from $k$. 

Now assume that $f \in \cH_c$. It then follows from the conformality of $\mu$ that 
\begin{align} 
\nonumber 	
\llbracket  \cL^k_{\varphi}(f) \rrbracket_1  
& \leq  \sum_{v \in \cW^k}  \llbracket  \Phi_k \circ \tau_v  f\circ \tau_v \rrbracket_1  
=  \sum_{v \in \cW^k}  \left( {\textstyle \sum_{g\in G} (\int \Phi_k \circ \tau_v f\circ \tau_v(\cdot,g) d\mu)^2 }\right)^{1/2} \\
& = \sum_{v \in \cW^k}  \left( {\textstyle \sum_{g\in G} \mu([v])^2 f(x,g\psi_k(v)^{-1})^2 }\right)^{1/2} = \llbracket f \rrbracket_1 ,  
\end{align}
where $x \in \Sig_A$ is arbitrary. Hence, $\Lambda_k \leq 1$ for all $k\in \N$. Finally, observe that the Gibbs property of $\mu$ implies that
\[ \llbracket \cL^n_\varphi(\1_{X_{\id}}) \rrbracket_1 \geq  \|\cL^n_\varphi(\1_{X_{\id}})\|_1^{\id} = \sum_{v:\psi_n(v)=\id} \mu([v]) \gg  \sum_{x:T^n(x,\id)=(x,\id)} \Phi_n(x)  \geq \cZ_{a,\id}^{n}, \]
for all $a \in \cW^1$. Hence, $\limsup_{n \to \infty} (\Lambda_n)^{1/n} \geq \rho$.
~ \hfill $\square$ \end{proof}

The proof of the following result is inspired by an argument of Day (see \cite[Lemma 4]{Day:1964}) which relies on the rotundity of $\ell^2(G)$.  
Recall that a Banach space $(B,\|\cdot \|)$ is \emph{uniformly rotund} if for all $\delta>0$ there exists $\epsilon>0$ such that, for all $f,g$ with $\|f-g\| \geq \delta$ and $\|f\|=\|g\|=1$, it follows that $\|f+g\| \leq 2- \epsilon$. Note that the space $\cH$ does not has this property but the closed subspace $\cH_c$ which is isomorphic to  $\ell^2(G)$. Since $\ell^2(G)$ is uniformly rotund, it follows that  $\cH_c$ has this property as well.    

\begin{lemma}\label{lem:constant_almost_eigenfunction}
 Assume that $\rho=1$. Then there exists $n \in \N$ such that, for given $\epsilon >0$, there exists $f \in \cH_c$ with $f\geq 0$ and $\llbracket \cL_\varphi^n(f) - f \rrbracket_1  \leq \epsilon \llbracket f \rrbracket_1$.
\end{lemma}

\begin{proof} Let $k$ be given by $\mathcal{J}\subset \cW^{k}$ where $\mathcal{J}$ refers to the finite set given by Lemma \ref{lem:finite_subsets} and assume that  
\begin{equation} \nonumber \label{eq:no_almost_eigenfunction} \delta_1 := \inf\left\{\llbracket \cL_\varphi^k(f) - f \rrbracket_1 /\llbracket f \rrbracket_1 \with f \in \cH_c, f\geq 0, f \neq 0 \right\} >0. \end{equation}
\noindent \textsc{Step 1}. We begin showing that $\delta_1>0$ implies that there exists a finite set of words such that, for all $x \in \Sig_A$ and $f \in \cH_c$, the estimate (\ref{eq:decay_of_translations}) below holds and that two arbitrary words can be joined by elements of this finite set. In order to do so, note that it is possible to choose a finite set $\cW^\ast \subset \cW^k$ with 
$ \sum_{v \in \cW^k \setminus \cW^\ast} \mu([v]) \leq {\delta_1}/({4})$. 
For $f \in \mathcal{H}_c$, we hence have that 
\begin{align*}
\llbracket \sum_{v \in \cW^k \setminus \cW^\ast} \Phi_k\circ\tau_v f \circ\tau_v \rrbracket_1 
\leq \sum_{v \in \cW^k \setminus \cW^\ast} \mu([w]) \llbracket f \rrbracket_1 \leq {\delta_1 \llbracket f \rrbracket_1}/({4}).
\end{align*}
Using $L_\varphi(1)=1$ and the $\triangle$-inequality then gives 
\begin{align*}
\delta_1 \, \llbracket f \rrbracket_1& \leq \llbracket \cL_\varphi^k(f) - f \rrbracket_1 = 
\llbracket \sum_{v \in \cW^k} \Phi_k\circ\tau_v(f \circ\tau_v - f) \rrbracket_1 \\
& \leq  \sum_{v \in \cW^\ast} \llbracket \Phi_k\circ\tau_v(f \circ\tau_v - f) \rrbracket_1 + 
\llbracket \sum_{v \notin \cW^\ast}  \Phi_k\circ\tau_v f \circ\tau_v \rrbracket_1 +  \llbracket \sum_{v \notin \cW^\ast} \Phi_k\circ\tau_v f \rrbracket_1  \\
& \leq   \sum_{v \in \cW^\ast} \mu([v]) \llbracket (f\circ \tau_{v} - f)\1_{\te^k([v])\times G} \rrbracket_1 + \delta_1/2.
\end{align*}
It follows from a convex sum argument that there exists $w_f \in  \cW^\ast$ with 
$\llbracket (f\circ \tau_{w_f} - f)\1_{\te^k([w_f])\times G} \rrbracket_1 \geq \delta_1\llbracket f \rrbracket_1/2$. Furthermore, note that the  b.i.p.-property implies that $\mu(\te^k([v]))$ is uniformly bounded from below for all $v \in \cW^k$. 
With $\hat{f}(g)$ referring to $f(x,g)$, we hence obtain that    
\[ \llbracket (f\circ \tau_{w_f} - f)\1_{\te^k([w_f])\times G} \rrbracket_1 = \mu(\te^k([w_f])) \|\hat{f} - \hat{f}(\cdot \ \psi(w_f)^{-1}) \|_{\ell^2(G)} \gg \|\hat{f} - \hat{f}(\cdot \ \psi(w_f)^{-1}) \|_{\ell^2(G)}.\] 
For $x \in \Sig_A$ and $a \in \mathcal{I}_{\textrm{\tiny bip}}$ with $[a] \subset \te^k([w_f])$, we now choose $w_a, w_x \in \mathcal{J}$ such that $x \in \te^k([w_x])$, $(w_aa)$ is admissible, and $u:=(w_aw_x)$, $v:=(w_fw_x)$ are in $\cW^{2k}$. 
Since $\psi_k(w_a)=\psi_k(w_x)=\id$ and $f \in \mathcal{H}_c$, we have $\psi_{2k}(u) = \id$ and $\psi_{2k}(v) = \psi_{k}(w_f)$. Hence, the  rotundity of $\cH_c$ implies that there exists a uniform constant ${\delta_2}>0$ with    
\begin{equation} \label{eq:decay_of_translations}
\frac{1}{2}\left\|  f \circ \tau_{u}(x,\cdot\ ) +   f\circ \tau_{v}(x,\cdot \ ) \right\|_{\ell^2(G)} 
\leq  (1-{\delta_2})\llbracket f \rrbracket_1, \ \forall x \in [a].
 \end{equation}
Furthermore, we may choose $w_a \in \mathcal{J}$ above such that there exists $\beta \in \mathcal{I}_{\textrm{\tiny bip}}$ with $ [w_a] \cup [w_f] \subset \te([\beta])$. By substituting $u$ and $v$ with $(w u)$ and  $(wv)$, respectively, for some admissible word $w$ in $\mathcal{J}^{m-2}$ and $m \in \N$ to be specified later, we hence may assume that there exists a finite subset $\cW^\dagger$ of $\cW^{mk}$ with the following properties. For all $w_i \in \cW^{n_i}$, $i=1,2$ and $f \in \cH_c$, there exist $u({w_1,f,w_2})$ and $v({w_1,f,w_2})$ in $\cW^\dagger$ such that
\begin{enumerate}
 \item  the estimate (\ref{eq:decay_of_translations}) holds for $u := u({w_1,f,w_2})$ and $v := v({w_1,f,w_2})$, with $x \in [w_2]$,
 \item  the first $k(m-2)$ letters of $ u({w_1,f,w_2})$ and $ v({w_1,f,w_2})$ coincide,
 \item  $w_1 u({w_1,f,w_2}) w_2$ and $w_1 v({w_1,f,w_2}) w_2$ are admissible. 
\end{enumerate}
\noindent \textsc{Step 2}. In order to obtain the estimate in (\ref{eq:decay_of_translations_v2}), we prove that the fluctuations of the potential can be absorbed into the uniform factor in (\ref{eq:decay_of_translations}). Observe that it follows from the H\"older continuity of $\log \varphi$ that we may choose $m$ such that 
\begin{equation}\label{eq:absorbing_hoelder_coeficients}
  \sqrt{1-\delta_2}  \leq \frac{\Phi_n \circ \tau_{w_1}(\tau_{u({w_1,f,w_2})}(x))}{\Phi_n \circ \tau_{w_1}(\tau_{v({w_1,f,w_2})}(x))} \leq \left(\sqrt{1-\delta_2} \right)^{-1},\end{equation}
 for all $n \in \N$ and $w \in \cW^n$. 
 Moreover, since $|\cW^\dagger| < \infty$, we have 
 \[2 \alpha := \inf\left(\left\{ \Phi_{mk}(\tau_{u}(x)) \with x \in \te^{mk}([u]), u \in \cW^\dagger \right\}\right)>0.\]
 By dividing each $u \in \cW^\dagger$ into two words $u_1$ and $u_2$ and setting $\Phi_{mk}(\tau_{u_2}(x)) := \Phi_{mk}(\tau_{u}(x)) - \alpha$  and $\Phi_{mk}(\tau_{u_1}(x)):= \alpha$ for each  $x \in \te^{mk}([u])$, we may assume 
without loss of generality that $\Phi_{mk}(\tau_u(x)) = \alpha$ for all $x \in \te^{mk}([u])$ and $u \in \cW^\dagger$.

By combining the above considerations, we are now in position to prove the main estimate with respect to $\cW^\dagger$, for a given function $f \in \cH_c$ and $(k+1)$ finite words $w_i \in \cW^{n_i}$, for $i=0,1,\ldots,k$. For a finite word $w$, set $f_w(x,g):=f(\tau_w(x), g \psi(w)^{-1})$, and define by induction, for $j=1,2, \ldots,k$,
\begin{align*} f_{j} & := \sum_{(i_1,\ldots,i_{j-1}) \in \{1,2\}^{j-1}} f_{w_0 u_1^{(i_1)} w_1 \ldots u_{j-1}^{(i_{j-1})} w_{j-1}},\\    
u_j^{(1)} & := u(w_{j-1},f_j,w_j), \quad u_j^{(2)} := v(w_{j-1},f_j,w_j).
\end{align*}
We then have, where $N := n_0 + n_1 + \ldots + n_n + nmk$, 
\begin{align*}  
 & \big\| 
 \sum_{(i_1,\ldots,i_n) \in \{1,2\}^n} \Phi_N(\tau_{w_0 u^{(i_1)}_1 w_1 \ldots u^{(i_n)}_n w_n}(x)) f(\tau_{w_0 u^{(i_1)}_1 w_1 \ldots u^{(i_n)}_n w_n})(x,\cdot \ )  \big\|_{\ell^2(G)} \\
 & \leq \ \left(\max_{(i_1,\ldots,i_n) \in \{1,2\}^n} \Phi_{N}(\tau_{w_0 u^{(i_1)}_1\ldots  u^{(i_n)}_n w_n}(x))\right)  \big\| \sum_{(i_1,\ldots,i_n) \in \{1,2\}^n}  \hat{f}_{w_0 u^{(i_1)}_1 \ldots w_{n-1}}(\cdot\   \psi(w_n)^{-1}      \psi({u^{(i_n)}_n})^{-1} )  \big\|_{\ell^2(G)} \\
 & = \ \left(\max_{(i_1,\ldots,i_n) \in \{1,2\}^n} \Phi_{N}(\tau_{w_0 u^{(i_1)}_1 \ldots  u^{(i_n)}_n w_n}(x))\right)  
\big\| \sum_{i_n \in \{1,2\}}  \hat{f}_{n}(\cdot\     \psi(({u^{(i_{n})}_{n}}))^{-1} )  \big\|_{\ell^2(G)} \\
& \leq \ \left(\max_{(i_1,\ldots,i_n) \in \{1,2\}^n} \Phi_{N}(\tau_{w_0 u^{(i_1)}_1 \ldots  u^{(i_n)}_n w_n}(x))\right)  
\left( 2(1-\delta_2)\right) \big\| \sum_{i_{n-1}\in \{1,2\}}  \hat{f}_{n-1}(\cdot\   \psi({u^{(i_{n-1})}_{n-1}})^{-1} )  \big\|_{\ell^2(G)}. 
\end{align*}
Combing (\ref{eq:absorbing_hoelder_coeficients}) with $\Phi_{mk}|_{[u]}=\alpha$ for all $u \in \cW^\dagger$, we hence obtain that, with $\delta_3:= 1- \sqrt{1-\delta_2}$, 
\begin{align}  
\nonumber
 & \big\| 
 \sum_{(i_1,\ldots,i_n) \in \{1,2\}^n} \Phi_N(\tau_{w_0 u^{(i_1)}_1 w_1 \ldots u^{(i_n)}_n w_n}(x)) f(\tau_{w_0 u^{(i_1)}_1 w_1 \ldots u^{(i_n)}_n w_n})(x,\cdot \ ) ) \big\|_{\ell^2(G)} \\
\nonumber
& \leq \ \left(\max_{(i_1,\ldots,i_n) \in \{1,2\}^n} \Phi_{N}(\tau_{w_0 u^{(i_1)}_1 w_1\ldots  u^{(i_n)}_n w_n}(x))\right)  
\left( 2(1-\delta_2)\right)^n \llbracket f \rrbracket_1 \\
\label{eq:decay_of_translations_v2}
& \leq \  \left( \sum_{(i_1,\ldots,i_n) \in \{1,2\}^n} \Phi_{N}(\tau_{w_0 u^{(i_1)}_1 w_1\ldots  u^{(i_n)}_n w_n}(x))  \right) \left(1-\delta_3\right)^{n} \llbracket f \rrbracket_1.
\end{align} 

\noindent \textsc{Step 3}. We  prove that $\delta_1>0$ implies that $\|\cL_\varphi^{nmk}(f)(x,\cdot )\|_{\ell^2(G)}$ decays exponentially and show that this is a contradiction to the last statement in Proposition \ref{prop:operator_on_cH}.

We now fix $f \in \cH_c$ and $x \in \Sig_A$. For $n \in \N$, we refer to $\cF$ as the set of all subsets of $\{1,2,\ldots, n\}$ and define, for each $\omega:= \{k_1, k_2,\ldots, k_d\} \in \cF$, a subset $\cV_\omega$ of $\cW^{nmk}$  as follows. A word $w=(w_1\ldots w_n) \in \cW^{nmk}$ is an element of $\cV_\omega$ if and only if there exist $w^{(i_j)}_{k_j} \in \cW^\dagger$, for $i_j=1,2$ and $j=1, \ldots,d$, such that $w_{k_j}= w^{(1)}_{k_j}$ and
\[ \big\|  \sum_{\star} \Phi_{nmk}(\tau_{v}(x)) f(\tau_{v}(x,\cdot)) \big\|_{\ell^2(G)} \leq \left( \sum_{\star} \Phi_{nmk}(\tau_{v}(x)) \right) (1-\delta_3)^{|\omega|} \llbracket f \rrbracket_1,
\]
where $\star$ stands for the summation over all $v = (v_1\ldots v_n)$ with $v_i = w_i$ for $i \notin \omega$ and 
$v_i \in \{w_i^{(1)},w_i^{(2)}\}$ or $i \in \omega$. Observe that the construction of $u_{i}^{(1)}$ and $u_{i}^{(2)}$ above and the estimate (\ref{eq:decay_of_translations_v2}) imply that $\{\cV_\omega \with \omega \in \cF\}$ is a covering of $\cW^{nmk}$. Hence,
\[ \cV^\ast_\omega := \cV_\omega\setminus\left({\textstyle \bigcup_{\omega \subset \omega', \omega \neq \omega'} \cV_{\omega'}}  \right) \]
defines a partition of $\cW^{nmk}$. Furthermore, for $\omega:= \{k_1, k_2,\ldots, k_d\}$, $j \in \{1,\ldots,d\}$ such that $k_j -1 \notin \omega$ and $w_1 \ldots w_n \in \cV^\ast_\omega$, we have
\[\sum_{v_1 \ldots v_{k_j -1} w_{k_j} \ldots w_n \in \cV^\ast_\omega} \Phi_{mk}(\tau_{v_{k_j -1}w_{k_j} \ldots w_n}(x)) \leq 1 - 2\alpha,\]
since there exists at least one pair of elements $w^{(i)}$ ($i=1,2$) in $\cW^\dagger$ such that $w_1\ldots w^{(i)}w_{k_j} \ldots w_n \in \cV^\ast_{\omega'}$ with $\omega' = \omega \cup \{k_j -1\}$. Hence, 
\begin{align*}
\| \cL_\varphi^{nmk}(f)(x,\cdot\ )\|_{\ell^2(G)} & = \big\| \sum_{\omega \in \cF} \sum_{w \in \cV^\ast_\omega} \Phi_{{nmk}}(\tau_w(x)) \hat{f}(\cdot \ \psi_w^{-1}) \big\|_{\ell^2(G)} \\
& \leq  \sum_{\omega \in \cF} (1-2\alpha)^{n-|\omega|} (2\alpha)^{|\omega|} (1-\delta_3)^{|\omega|}  \llbracket f \rrbracket_1 \\
& = \sum_{j=0}^n \genfrac{(}{)}{0pt}{}{n}{k} (1-2\alpha)^{n-k}  (2\alpha(1-\delta_3))^{k} \llbracket f \rrbracket_1= (1 - 2\alpha \delta)^n \llbracket f \rrbracket_1.
 \end{align*}
By applying Jensen's inequality, we obtain $\llbracket \cL_\varphi^{nmk}(f)\rrbracket_1 \leq (1 - 2\alpha \delta)^n \llbracket f \rrbracket_1$ which is a contradiction to $\rho=1$ by Proposition \ref{prop:operator_on_cH}. Hence, $\delta_1=0$. 
~ \hfill $\square$ \end{proof}

We are now in position to prove the converse to Theorem \ref{theo:Kesten_pt1}. 

\begin{theorem}\label{theo:Kesten_reverse}  Assume that $(\Sig_A,\te)$ is a topological Markov chain with the b.i.p.-property, that $(X,T)$ is a topologically transitive $G$-extension and that $\varphi$ is a Hölder continuous potential with $\|L_\varphi(1)\|_\infty < \infty$. Then $P_G(T,\varphi)=P_G(\te,\varphi)$ implies that the group $G$ is amenable.
\end{theorem}

\begin{proof} We assume without loss of generality that $P_G(\te,\varphi)=0$.  Now choose a finite subset $K$ of $G$. It then follows by decomposition of $T$ into mixing components that there exists $m \in \N$ such that $m$ is a multiple of $n$ in Lemma \ref{lem:constant_almost_eigenfunction} and
$K \subset \{ \psi(v) : v \in \cW^{m}\}$. In particular, there exists a finite subset $\cW_K$ of $\cW^m$ with 
\[ K = \{  \psi(v) : v \in \cW_K \}.\]  
By Lemma \ref{lem:constant_almost_eigenfunction}, there exists a sequence of positive functions $(f_k)$ in $\cH_c$ with  $\lim_{k\to \infty} \llbracket \cL^m(f_k)-f_k \rrbracket_1 =0 $, $f_k \geq 0$  and $\llbracket f_k \rrbracket_1 =1$ for all $k \in \N$. Note that this, in particular, implies that $\lim_{k\to \infty} \llbracket \cL^{m}(f_k) \rrbracket_1 =1$. 
In order to show that $\llbracket (f_k \circ \tau_v - f_k)\cdot \1_{\te^m([v])\times G} \rrbracket_1 \to 0$ for all $v \in \cW_K$, assume that there exists $v \in \cW_K$ with $ \liminf_k \llbracket (f_k \circ \tau_v - f_k)\cdot \1_{\te^m([v])\times G} \rrbracket_1 >0$. By the argument in the first step of proof of Lemma \ref{lem:constant_almost_eigenfunction}, this implies that $ \llbracket \cL^{m}(f_k) \rrbracket_1 $ is bounded away from 1, which is a contradiction. Hence, $ \liminf_k \llbracket (f_k \circ \tau_v - f_k)\cdot \1_{[\te^m([v]),\cdot]} \rrbracket_1 =0$ for all $v \in \cW_K$, and, by taking a subsequence, we may assume without loss of generality that
\[ \lim_{k\to \infty} \llbracket (f_k \circ \tau_v - f_k)\cdot \1_{{\te^m([v])\times G}} \rrbracket_1 =0,  \quad \forall v \in \cW_K.\]
Furthermore, recall that the Hölder inequality implies, with $\hat{f}_k(g):= f_k(x,g)$ and $h\in G$, that
\begin{align}
\nonumber
\|\hat{f_k}^2(\cdot) - \hat{f_k}^2(\cdot h)\|_1 & = \sum_{g \in G}|\hat{f_k}^2(g) - \hat{f_k}^2(gh)| 
= \sum_{g \in G}|\hat{f_k}(g) - \hat{f_k}(gh)| \cdot |\hat{f_k}(g) + \hat{f_k}(gh)|  \\
\label{eq:l_1-estimate}
& \leq \|\hat{f_k}(\cdot) - \hat{f_k}(\cdot h)\|_2 \cdot  \|\hat{f_k}(\cdot) + \hat{f_k}(\cdot h)\|_2 \leq  2\|\hat{f_k}(\cdot) - \hat{f_k}(\cdot h)\|_2.
\end{align}
We now fix $k$ to be specified later and use the following representation of $\hat{f_k}^2$. There exist $p \in \N \cup \{\infty\}$ and  $\lambda_i> 0 $ and $A_i \subset G$ with $A_i \subset A_{i+1}$, for $1\leq i<p$, such that $\hat{f_k}^2 = \sum_{i=1}^p \lambda_i \1_{A_i}$. In particular, observe that $\sum_{i=1}^p \lambda_i |{A_i}|=1$ and that   $(A_i h \setminus A_i) \cap  (A_j    \setminus A_j h) = \emptyset$ by  monotonicity of $(A_i)$. Hence, 
\begin{align}\nonumber
\left\|\hat{f_k}^2(\cdot) - \hat{f_k}^2(\cdot h)\right\|_1 
& =  \sum_{g \in G} \left|{\textstyle \sum_i \lambda_i (\1_{A_i}(g) - \1_{A_ih^{-1}}(g))}\right|\\
\label{eq:no_absolute_value}
& =  \sum_{g \in G} { \sum_{i=1}^p \lambda_i \1_{A_i  \triangle A_i h^{-1}}(g)} = \sum_{i=1}^p \lambda_i \left|A_i  \triangle A_ih^{-1}\right|.
\end{align}
We are now in position to prove the amenability of $G$. For $\epsilon >0$, choose $k$ such that
\[ \llbracket (f_k \circ \tau_v - f_k)\cdot \1_{[\te^m(v),\cdot]} \rrbracket =  \|\hat{f_k}(\cdot) - \hat{f_k}(\cdot \psi_m(v)^{-1})\|_2  \leq \epsilon/|\cW_K|  \]
for all $v \in \cW_K$. Combining estimate (\ref{eq:l_1-estimate}) and the identity (\ref{eq:no_absolute_value}) then implies that
\[ \epsilon \geq \frac{1}{2} \sum_{i=1}^p \lambda_i \sum_{h \in K} \left|A_i h \triangle A_i\right|.\]
Hence, there exists $1\leq i \leq p$, such that $\sum_{h \in K} |A_i h \triangle A_i| \leq 2\epsilon|A_i|$.
Note that the above argument shows that for each finite set $K$ and $\epsilon > 0$, there exists a  $(K,\epsilon)$-F\o lner set $A$, that is $A$ is finite and   
\[\sum_{h \in K} |A h \triangle A| \leq \epsilon|A|.\]
In order to prove amenability of $G$ through the construction of a F\o lner sequence $(B_n)$, choose a sequence of finite sets $(K_n)$ with $K_n \nearrow G$ and assume, by induction, that $B_{n+1}$ is a $(K_n\cup B_n,1/n)$-F\o lner set. It is then easy to see that 
$\lim |g B_n \triangle B_n|/ |B_n|=0$ for all $g \in G$.
~ \hfill $\square$ \end{proof} 

It is worth noting that the proof is inspired by an argument in \cite{Greenleaf:1969} where the identity (\ref{eq:no_absolute_value}) was used to derive a weak F\o lner condition. However, there is an alternative chain of arguments for the proof. For $n\in \N$ and $f \in \ell^2(G)$, set $Q_n(f)(g):=\sum_{v \in \cW}\mu([v])f(g\psi_v^{-1})$. It follows from Lemma \ref{lem:constant_almost_eigenfunction} that $\|Q_n\|=1$. Hence, by a result of Day (\cite[Theorem 1(d)]{Day:1964}), $G$ is amenable.
We now present two immediate applications of our results to the co-growth of groups and group extensions of Gibbs-Markov maps.

\paragraph{Co-growth} As a corollary of Theorems \ref{theo:Kesten_pt1} and \ref{theo:Kesten_reverse}, we obtain the criteria for amenability in terms of the co-growth as introduced in \cite{Grigorchuk:1980}. For a set of generators $\mathcal{G}=\{\gamma_1, \gamma_1^{-1}, \ldots, \gamma_r^{-1}\}$ of $G$, let $\Sig_A$ be the subshift of finite type with $\cW=\mathcal{G}$ and transition matrix $(a_{gh})$ given by $a_{gh}=0$ if and only if $(g,h)$ is equal to $(\gamma_i,\gamma_1^{-1})$ or $(\gamma_i^{-1},\gamma_1)$ for some $i=1,\ldots,r$. With respect to the potential $1$ and the cocycle $\psi|_{[g]}=g$, the Grigorchuk-Cohen criterion follows from $P_G(\te)=\log(|\mathcal{G}|-1)$. That is, $G$ is amenable if and only if   
 \[ \limsup_{n \to \infty} |\{w \in \cW^n \with \psi_n(w)=\id \}|^{1/n} =  |\mathcal{G}|-1=2r-1.\]

\paragraph{Gibbs-Markov maps} The class of Gibbs-Markov maps was implicitly introduced in \cite{AaronsonDenkerUrbanski:1993}, and we recall its  definition now. Let $\mu$ be a Borel probability measure on $\Sig_A$ such that, for all $w \in \cW^1$, $\mu$ and $\mu \circ \tau_w$ are equivalent. Then $(\Sig_A,\te, \mu)$ is called  a \emph{Gibbs-Markov map} if 
$\inf \{\mu(\te([w]))\with w \in \cW^1\} >0$ and 
there exists $0<r<1$ such that, 
for all $m,n \in \N$, $v \in \cW^m$, $w \in \cW^n$ with $(vw)\in \cW^{m+n}$,
\begin{equation} \label{eq:GM} \sup_{x,y \in [w]} \left|  \log \frac{d\mu \circ \tau_v}{d\mu}(x) -  \log\frac{d\mu \circ \tau_v}{d\mu}(y) \right| \ll r^n. \end{equation}
Furthermore, we refer to a group extension of a Gibbs-Markov map as \emph{weakly symmetric} if $(\Sig_A,\te,\psi)$ is symmetric and  
\[ \lim_{n\to \infty} \sup_{w \in \cW^n} \sqrt[n]{{\mu([w])}/{\mu([w^\dagger]}} = 1.\]
\begin{theorem}
Assume that $(\Sig_A,\te, \mu)$ is a topologically mixing Gibbs-Markov map with the b.i.p.-property such that $(X,T)$ is a topologically transitive $G$-extension.
\begin{enumerate}
\item If the group extension is weakly symmetric and $G$ is amenable, then
\begin{equation} \label{cond:measure_decay_t} \limsup_{n \to \infty} (\mu(\{x \in \Sig_A \with \psi_n(x)=\id \})^{1/n} =1. \end{equation}  
\item If (\ref{cond:measure_decay_t}) holds, then $G$ is amenable.
\end{enumerate} 
\end{theorem}
\begin{proof} Set $\Phi_n := {d\mu} / d\mu\circ\theta$, for $n \in \N$. Since $({d\mu} / d\mu\circ\theta)(x) = ({d\mu\circ \tau_v}/{d\mu})(\theta^n(x)$, for $x \in [v]$ and $v \in \cW^n$, observe that the Gibbs-Markov property in (\ref{eq:GM}) implies that $\Phi_n$ is of bounded variation and $\log \Phi_n$ is locally H\"older continuous. Furthermore, combining the last estimate in the proof of Lemma \ref{prop:operator_on_cH} and the finiteness of  $\mathcal{J}\subset \cW^k$ of Lemma \ref{lem:finite_subsets}, we obtain
\[ \cZ^{n+ k}_{a,\id} \gg \mu(\{x \with \psi_n(x)=\id \})\gg \cZ^n_{a,\id}, \]
for all $a \in \cW^1$ and $n\in \N$. Hence, $P_G(T)= \limsup_{n} (1/n) \log \mu(\{x : \psi_n(x)=\id \})$ and the result follows from Theorems \ref{theo:Kesten_pt1} and \ref{theo:Kesten_reverse}.
~ \hfill $\square$ \end{proof}

Note that the above result is an extension of Kesten's result to a class of random walks on groups with stationary increments.

%

\section{An application to hyperbolic geometry}
In order to apply the above results to normal covers of hyperbolic manifolds, we recall the following definition from \cite{StadlbauerStratmann:2005}. A Kleinian group is called essentially free if there exists a Poincar\'e fundamental polyhedron $F$ with faces $f_1, f_2, \ldots f_{2n}$ and associated generators $g_1,g_2 \ldots g_{2n}$ of $G$ with $g_{i}(f_i) = f_{i+n}$, $g^{-1}_{i}(f_{i+n}) = f_{i}$ and $g_i^{-1}=g_{i+n}$ for $i=1,\ldots n$, such that the following conditions are satisfied. In here, we refer to  $\overline{(\cdot)}_{\overline{\H}}$ as the closure in $\overline{\H}$. 
\begin{enumerate} 
\item If $\overline{(f_i)}_{\overline{\H}} \cap \overline{(\bigcup_{j\neq i}f_j)}_{\overline{\H}} \neq \emptyset $ for some $i =1,2,\ldots n$, then  $g_i,g_{i+n}$ are hyberbolic transformations, and $\overline{(f_{i+n})}_{\overline{\H}} \cap \overline{(\bigcup_{j\neq i+n}f_j)}_{\overline{\H}} \neq \emptyset $,    
\item if $\overline{(f_i)}_{\overline{\H}}\cap \overline{(f_j)}_{\overline{\H}}$ is a single point $p$ for some $j =1,2,\ldots 2n$, then $p$ is a parabolic fixed point,
\item if $f_i \cap f_j \neq \emptyset$ for some $j =1,2,\ldots 2n$, then $g_ig_j = g_jg_i$.
\end{enumerate}
Observe that this class comprises all non-cocompact, geometrically finite Fuchsian groups, the class of Schottky groups, and in general gives rise to geometrically finite hyperbolic manifolds which may have cusps of arbitrary rank.

We now proceed with the construction of the associated coding map. Fix a point  $\mathbf{o}$ in the interior $F$, and denote by $a_i$ the intersection of the shadow of $f_i$ and the radial limit set $L_r(G)$ of $G$. 
Furthermore, denote by $\kappa$ the inversion, acting on $\{f_i:i=1,\ldots 2n\}$, $\{g_i:i=1,\ldots, 2n\}$ and $\{a_i:i=1,\ldots, 2n\}$, which is defined by $\kappa f_i=f_{i+n}$, $\kappa g_i=g_{i+n}$ and $\kappa a_i=a_{i+n}$, respectively. Now let $a$ be an atom of the partition $\alpha$ generated by $\{a_i:i=1,\ldots ,2n\}$. Hence there exist $1 \leq k\leq 2n$ and $i_1, \ldots i_k \in \{1, \ldots, 2n\}$ such that $a = \bigcap_{i=1}^k a_{i_l}$. Choose $g_a \in \{g_{i_l}: k= 1,\ldots k\}$, and for $\kappa a:= \bigcap_{i=1}^k \kappa a_{i_l}$, set $g_{\kappa a} = g_a^{-1}$. This then gives rise to the coding map $\Gamma$ defined by 
\[\Gamma: L_r(G) \to L_r(G), \te(x)=g_a(x) \hbox{ for } x \in a, a \in \alpha\]
and a canonical symmetry given by $a^\dagger:=\kappa a$. 
For further details of this construction, we refer to \cite{StadlbauerStratmann:2005}, where it is shown, that $\te$ is well defined,  $\alpha$ is a Markov partition and  the underlying subshift of finite type is topologically mixing. In particular, $L_r(G)$ can be identified with a shift space $\Sig_A$ and $\Gamma$ with the one-sided shift map.

 In order to specify a potential adapted to the geometry of $\H$, recall that  the  Poisson kernel $\cK$ with respect to the ball model is given by, for $x\in \partial \H$ and $\mathbf{o} \in   \H$,
\[\cK(\mathbf{o},x):= \frac{1-|g(\mathbf{o})|^2}{|g(\mathbf{o})-x|^2}.\]
It is well known that $\log \cK(g(\mathbf{o}),x)$, for $g\in G$, is equal to the orientated hyperbolic distance between $\mathbf{o}$ and the horocycle through $g(\mathbf{o})$ and $x$. Note that this horocyclic distance sometimes also is referred to as the Busemann cocycle. 
Furthermore, we recall the following for the potential given by
 \begin{equation}\label{def:Kernel}\varphi(x) = (\mathcal{K}(g_a(\mathbf{o}),x))^\delta,
 \end{equation} 
for $x \in a, a \in \alpha$ and $\delta>0$  to be specified later. If $G$ is convex cocompact, then  $\log \varphi$ is Hölder continuous with respect to the shift metric induced by $\Sig_A$. In case that $G$ is essentially free and contains parabolic elements, set 
\[P:= \overline{F}\cap \{p \in L(G) \with p \hbox{ is a fixed point of a parabolic element in } G \}.\] 
Then, if $B$ is a subset of $L_r(G)$ which is bounded away from $P$ and measurable with respect to some finite refinement by preimages of $\alpha$, the potential associated with the first return map to $B$ is Hölder continuous (see \cite[Lemmata 3.3 \& 3.4]{KessebohmerStratmann:2004a}). For ease of notation, let $(\Sig,\te)$ refer to the first return map to $B$ if $G$ contains parabolic elements and to $(\Sig_A,\Gamma)$ if $G$ does not contain parabolic elements. Observe that in both cases the potential  is Hölder continuous with respect to the shift metric, $(\Sig,\te)$ has the b.i.p.-property and can be identified with a maximal non-invertible factor of a Poincaré section of the geodesic flow on the unit tangent bundle  $T^1(\mathbb{H}/G)$  of $\mathbb{H}/G$ (see \cite{StadlbauerStratmann:2005}). Furthermore, since by construction $B$ is bounded away from $P$, each return to the associated Poincaré section in  $T^1(\mathbb{H}/G)$ corresponds to a return to a ball of bounded diameter with center $\mathbf{o}$ in $\mathbb{H}/G$.  

Combining these observations, one then obtains the following relation between the finite words  of $\Sig$, the elements of $G$ and the hyperbolic distance $d(\mathbf{o},g(\mathbf{o}))$. Let $\cW^n$ refer to the words of length $n$ of $\Sig$ and $g_w \in G$ to the element in $G$ defined by $\te^n|_{[w]} = g_w$. Note that, by construction of $\Gamma$, $G$ is isomorphic to the finite words with respect to the original partition $\alpha$ (see \cite{StadlbauerStratmann:2005}). Hence, after a possible induction to $B$, the map $\cW^\infty \to G$,  $w \mapsto g_w$ is injective. Using the property of bounded returns, it follows that the map is almost onto in the sense that there exists a finite subset $J$ of $G$ such that
\begin{equation} \label{eq:almost_onto} G = \bigcup_{h \in J,w \in \cW^\infty} hg_w.\end{equation} 
For $w \in \cW^\infty$, set $ \Phi_w:=\sup_{x \in [w]}\Phi_{|w|}(x)$. As a further consequence of bounded returns, it then follows that, for $s>0$, 
\begin{equation}\label{eq:comparable}
\Phi_w^s \asymp e^{-s\delta d(\mathbf{o},g_w(\mathbf{o}))}.
\end{equation}
Now assume that $N$ is a normal subgroup of $G$ and recall that in this situation, the manifold $\H/N$ is called periodic with period $G/N$. Since $\H/N$ is a cover of $\H/G$, it follows that $\H/N$ is  geometrically finite if and only if $G/N$ is finite.  The properties (\ref{eq:comparable}) and (\ref{eq:almost_onto}) above now allow relating the exponents of convergence of $N$ and $G$ in terms of the amenability of $G/N$. Therefore, recall that the Dirichlet series
\[
\cP(H,s) := \sum_{g \in H} e^{-s d(\mathbf{o},g(\mathbf{o}))} 
\]
is referred to as the Poincaré series of the Kleinian group $H$, and that its abzissa of convergence $\delta(H)$ is called the exponent of convergence of $H$. In order to quantify $\cP(N,s)$, we will now employ our results on group extensions. In order to do so, for $w \in \cW$ and $x \in [w]$, set $\psi(x):= [g_w] \in G/N$ and   
\[ T: \Sig\times G/N \to \Sig\times G/N, \quad (x,[g]) \mapsto  (\te(x),[g]\psi(x)).\]
In here, we only will make use of the estimates on $\cL^n(\1_{X_{\id}})$, but it is worth noting that $T$ is related to the geodesic flow on the periodic manifold. That is, $T$ is a non-invertible factor of the base transformation of a Ambrose-Kakutani representation of the flow on the periodic manifold (see also, e.g. \cite{AaronsonDenker:1999,Sarig:2004}), where the associated measure is the Liouville-Patterson measure of $G$. 

As an application of Theorems \ref{theo:Kesten_pt1} and \ref{theo:Kesten_reverse} we now obtain the following partial refinement of a result of Brooks for convex-cocompact Kleinian groups in \cite{Brooks:1985} to the class of essentially free Kleinian groups.

\begin{theorem} \label{theo:brooks}
Let $G$ be an essentially free Kleinian group and $N \lhd G$ a normal subgroup. Then $\delta(G)= \delta(N)$ if and only if $G/N$ is amenable.  
\end{theorem}

\begin{proof} Set $\delta= \delta(G)$ in the definition of $\varphi$ and note that for this choice, the existence of a finite, invariant measure implies that $P_G(\te,\varphi)=0$ (see, e.g., \cite{StadlbauerStratmann:2005}). Furthermore, as a consequence of (\ref{eq:comparable}) we have that $\varphi$ is symmetric with respect to the involution generated by $\kappa$. Finally, it can easily be deduced from the connection between the group extension $T$ and the geodesic flow on $T^1(\H/N)$ that $T$ is topologically transitive. Hence, Theorems \ref{theo:Kesten_pt1} and \ref{theo:Kesten_reverse} are applicable and therefore, it remains to show that $P_G(T)< P_G(\te)$ if and only if $\delta(N)<\delta=\delta(G)$. 

So assume that $\delta(N)<\delta$. Hence, for $\epsilon>0$ with  $\delta(N) < (1-\epsilon)\delta$, $\cP(N,(1-\epsilon)\delta)< \infty$. In particular, applying (\ref{eq:almost_onto}) and (\ref{eq:comparable}) gives that
\begin{align*}
\infty > \sum_{g \in N} e^{- (1-\epsilon)\delta(\mathbf{o},g(\mathbf{o}))} \asymp 
\sum_{\genfrac{}{}{0pt}{}{w \in \cW^\infty,}{ [g_w]=\id}} \Phi_w^{1-\epsilon} 
\geq  \sum_{\genfrac{}{}{0pt}{}{w \in \cW^\infty,}{ [g_w]=\id}} 
\Phi_w \|\varphi\|_\infty^{-\epsilon |w|}.
\end{align*}
Since  $\exp(-P_G(T))$ is the radius of convergence of $\sum_{[g_w]=\id} \Phi_w x^{|w|}$, it follows from $\|\varphi\|_\infty < 1$ that $P_G(T) \leq \epsilon \log \|\varphi\|_\infty < 0 = P_G(\te)$.

Now assume that $P_G(T)<0$ and $G$ does not contain parabolic elements. Then $\cW$ is finite and, in particular, $\|1/\varphi\|_\infty < \infty$. Since $P_G(T)<0$, we have that, for $1<x<\exp(-P_G(T))$,  
\begin{align}  \label{eq:convex_cocompact_P<0}
 \infty  > \sum_{\genfrac{}{}{0pt}{}{w \in \cW^\infty,}{ [g_w]=\id}} \Phi_w x^{|w|} \geq \sum_{\genfrac{}{}{0pt}{}{w \in \cW^\infty,}{ [g_w]=\id}}  \Phi_w^{1-\epsilon} {x^{|w|}}{\|1/\varphi\|_\infty^{-\epsilon|w|} }
  \asymp \sum_{\genfrac{}{}{0pt}{}{w \in \cW^\infty,}{ [g_w]=\id}}
 e^{- (1-\epsilon)\delta(\mathbf{o},g_w(\mathbf{o}))} {x^{|w|}}{\|1/\varphi\|_\infty^{-\epsilon|w|} }  . 
 \end{align} 
Hence, if $\epsilon 
< -P_G(T)/\log(\|1/\varphi\|_\infty)$, then $\cP(N,(1-\epsilon)\delta)< \infty$ and, in particular, $\delta(N)<\delta$.  

It remains to consider the case where $P_G(T)<0$ and $G$ contains parabolic elements. For $p \in P$, let $G_p$ denote the stabiliser of $p$ in $G$. 
 By well known arguments (see, e.g., Lemma 3.2 in \cite{StratmannVelani:1995}), we have, for $s > k_p/2$ and $\ell >0$, with $k_p$ referring to the Abelian rank of $G_p$, that 
\[ 
\sum_{\genfrac{}{}{0pt}{}{g \in G_p, }{ d(\mathbf{o},g(\mathbf{o}))\geq 2 \ell}} e^{-sd(\mathbf{o},g(\mathbf{o}))} \asymp \sum_{n \geq e^\ell} \frac{1}{n^{2s - k_p+1}} \asymp \frac{1}{2s - k_p} e^{\ell (k_p-2s)}. 
\]
Since $\delta(G)$ is always bigger than $k_p/2$, we may choose $0<\epsilon< 1-k_p(2\delta)^{-1}$.  For $s=(1-\epsilon)\delta $, we then have
\[
 \sum_{\genfrac{}{}{0pt}{}{g \in G_p, }{ d(\mathbf{o},g(\mathbf{o}))\geq 2 \ell}} e^{-(1-\epsilon)\delta d(\mathbf{o},g(\mathbf{o})}
\asymp 
e^{2 \ell \delta\epsilon}
\sum_{\genfrac{}{}{0pt}{}{g \in G_p, }{ d(\mathbf{o},g(\mathbf{o}))\geq 2 \ell}} e^{-\delta d(\mathbf{o},g(\mathbf{o}))} .
 \]
For $\Lambda >0$, set $\cW_\Lambda := \{w \in \cW^1\with \inf_{x \in [w]} \varphi(x) \geq \Lambda\}$. 
It follows from the inducing process of $\te$, that $g_w \in G_p$ for some  $p \in P$ and each $w \in \cW_\Lambda$, if $\Lambda>0$ is sufficiently small. The above estimate then implies the following uniform Lipschitz continuity. There exists $C\geq 1$ such that for arbitrary families $\{x_w \with x  \in [w], w \in \cW_\Lambda\}$, 
\[
 \frac{\sum_{w\in \cW_\Lambda}\varphi(x_w)^{1-\epsilon} }{\sum_{w\in \cW_\Lambda}\varphi(x_w)}  - 1  \leq C\epsilon.
\]
Combining the estimate with the argument in (\ref{eq:convex_cocompact_P<0}) then gives that
 $\delta(N)<\delta(G)$.
~ \hfill $\square$ \end{proof}

\section*{Acknowledgements}
The author acknowledges partial support by \emph{Fundação para Ciência e a Tecnologia} through grant SFRH/BPD/39195/2007 and the \emph{Centro de Matemática da Universidade do Porto}. Moreover, the author is indebted to Johannes Jaerisch for pointing out a gap in the proof of Lemma \ref{lem:constant_almost_eigenfunction} in an earlier version of the paper.


\end{document}